\documentclass[onefignum,onetabnum]{siamart171218}

\usepackage{amscd,amssymb}
\usepackage{amsfonts}
\usepackage{hyperref}
\usepackage{color}
\usepackage{amsmath}
\usepackage{amsfonts}
\usepackage{amssymb}
\usepackage{graphicx}
\usepackage{graphicx,color}
\usepackage{tikz}
\usepackage{wrapfig}
\usepackage{lipsum}
\usepackage{caption}
\usepackage{circuitikz}%
\usetikzlibrary{calc,patterns,angles,quotes}
\usetikzlibrary{calc,patterns,decorations.pathmorphing,decorations.markings}
\usetikzlibrary{calc,patterns,angles,quotes}
\usetikzlibrary{calc,patterns,decorations.pathmorphing,decorations.markings}
\newtheorem{example}[theorem]{Example}

\usepackage{lipsum}
\usepackage{amsfonts}
\usepackage{graphicx}
\usepackage{epstopdf}
\usepackage{algorithmic}
\ifpdf
  \DeclareGraphicsExtensions{.eps,.pdf,.png,.jpg}
\else
  \DeclareGraphicsExtensions{.eps}
\fi

\newsiamremark{remark}{Remark}
\newsiamremark{hypothesis}{Hypothesis}
\crefname{hypothesis}{Hypothesis}{Hypotheses}
\newsiamthm{claim}{Claim}

\headers{Invariance Pressure  of Control Sets}{F. Colonius, J. A. N. Cossich, and A. J. Santana}

\title{Invariance Pressure  of Control Sets\thanks{Submitted to the editors DATE.
\funding{This work was partly supported by
	Network NT8, OEA ICTP and and PICME/CAPES.}}}

\author{Fritz Colonius\thanks{Institut f\"{u}r Mathematik, Universit\"{a}t Augsburg, Augsburg, Germany 
  (\email{fritz.colonius@math.uni-augsburg.de}, \url{https://appa.math.uni-augsburg.de/Mitarbeiter/fc/}).}
\and Jo\~{a}o A.N. Cossich\thanks{Department of Mathematics, State University of Maringa, Maringa, Brazil 
  (\email{joaocossich\_@hotmail.com }, \email{ajsantana@uem.br}).}
\and Alexandre J. Santana\footnotemark[3]}

\usepackage{amsopn}

\ifpdf \hypersetup{
  pdftitle={Invariance Pressure  of Control Sets},
  pdfauthor={F. Colonius, J. A. N. Cossich, and A. J. Santana}
} \fi

\begin{document}

\maketitle

\begin{abstract}
  The invariance pressure of continuous-time control systems
with initial states in a set $K$ which are to be kept in a set $Q$
is introduced and a number of results are derived, mainly for the
case where $Q$ is a control set.
\end{abstract}



\section{Introduction}

This paper generalizes the notion of invariance pressure introduced by the
present authors in \cite{Cocosa} and improves its characterization for linear
control systems. We now admit initial states in a subset $K$ of the set $Q$ in
which the system should remain. Then we derive a number of results for special
sets $Q$, in particular, for control sets.

Invariance pressure is a generalization of invariance entropy for control
systems by introducing a potential, similarly as topological pressure
generalizes topological entropy for dynamical systems, cf. e.g. Walters
\cite{Walt82}. Basic references for invariance entropy include Nair, Evans,
Mareels and Moran \cite{NEMM04}, Colonius and Kawan \cite{ColKa09a} and the
monograph Kawan \cite{Kawa13}; cf. also Huang and Zhong \cite{HuanZ18} for
relations to dimension theory and Da Silva \cite{daSil14} for the case of
linear systems on Lie groups.. This concept, similarly as other entropy
concepts for control systems, like estimation entropy (Liberzon and Mitra
\cite{LibeM18}) and restoration entropy (Matveev and Pogromsky \cite{PogrM16})
are introduced in order to analyze the minimal data rates in control tasks
when data rate constraints are present (cf. also the monograph Matveev and
Savkin \cite{MatvS09}).

As main results we characterize\textbf{ }the invariance pressure for the
control set of linear control systems and for inner control sets we can show
that the limit superior in the definition of invariance pressure can be
replaced by the limit inferior.

The contents of this paper are as follows: After preliminaries on control
systems and definitions of invariance pressures for admissible pairs $(K,Q)$ in
the sense of Kawan \cite{Kawa13} in Section \ref{Section2}, Section 3 shows
several basic dynamical properties of invariance pressure, in particular, its
behavior under conjugacies. Section \ref{Section4} shows that the outer
invariance pressure and the invariance pressure coincide if the set $Q$ is
isolated, and the outer invariance pressure is discussed for sets $Q$ which
satisfy the no-return property.\ In particular, the invariance pressure is
independent of the choice of a compact subset $K$ with nonvoid interior in a
control set $D$. Section \ref{Section5} discusses the outer invariance
pressure for inner control sets. Finally, Section \ref{Section6} derives an
estimate for the invariance pressure of the control set of a linear control system.

\section{Preliminaries\label{Section2}}

In this section, we establish some notation and basic concepts for control
systems which will be used throughout the paper.

\subsection{Control systems}

A continuous-time control system on a smooth manifold $M$ is defined as a
system
\[
\Sigma\ :\ \dot{x}(t)=F(x(t),\omega(t)),\ \omega\in\mathcal{U},
\]
where $\mathcal{U}:=\{\omega:\mathbb{R}\rightarrow U;\ \omega
\mbox{ is measurable with }\omega(t)\in U\mbox{ almost everywhere}\}$ is a set
of \textbf{\textit{admissible control functions}}, the \textbf{\textit{control
range}} $U$ is a compact subset of $\mathbb{R}^{m}$ and $F:M\times
\mathbb{R}^{m}\rightarrow TM$ is a $C^{1}$-map such that for each $u\in U$,
$F_{u}(\cdot):=F(\cdot,u)$ is a vector field on $M$. For each $x\in M$ and
$\omega\in\mathcal{U}$, we suppose that there exists a unique solution
$\varphi(t,x,\omega)$ which is defined for all $t\in\mathbb{R}_{+}=[0,\infty
)$. We usually refer to the solution $\varphi(\cdot,x,\omega)$ as
\textbf{\textit{trajectory}} of $x$ with control function $\omega$. We also
fix a metric $\rho$ on $M$ which is compatible with the topology.

An important case in this paper is the linear control system $(A,B)$ on
$\mathbb{R}^{d}$, which is defined by
\[
\Sigma_{lin}\ :\dot{x}(t)=Ax(t)+B\omega(t),\ \omega\in\mathcal{U},
\]
where $A\in\mathbb{R}^{d\times d}$ and $B\in\mathbb{R}^{d\times m}$. We recall
that the solution of this system is given by
\[
\varphi(t,x,\omega)=e^{tA}x+\int_{0}^{t}e^{(t-s)A}B\omega(s)ds.
\]
We need several notions characterizing controllability properties of subsets
of the state space $M$ of system $\Sigma$.

The \textbf{\textit{positive}} and \textbf{\textit{negative orbits from}}
$x\in M$ are
\[
\mathcal{O}^{+}(x):=\{y\in M; \ \mbox{ there are } t>0 \mbox{ and } \omega
\in\mathcal{U} \mbox{ with } \varphi(t,x,\omega)=y\}
\]
and
\[
\mathcal{O}^{-}(x):=\{y\in M; \ \mbox{ there are } t>0 \mbox{ and } \omega
\in\mathcal{U} \mbox{ with } \varphi(t,y,\omega)=x\},
\]
respectively.

A set $Q\subset M$, is called \textbf{\textit{controlled invariant }}if for
all $x\in D$ there exists $\omega\in\mathcal{U}$ such that $\varphi
(t,x,\omega)\in Q$ for all $t\geq0$. We say that a set $Q\subset M$ satisfies
the \textbf{\textit{no-return property}}, if
\[
\forall x\in Q\ \forall\tau>0\ \forall\omega\in\mathcal{U}:\varphi
(\tau,x,\omega)\in Q\Rightarrow\varphi([0,\tau],x,\omega)\subset Q.
\]
A controlled invariant set $D\subset M$ is called a \textbf{\textit{control
set}} if satisfies $D\subset\overline{\mathcal{O}^{+}(x)}$ for all $x\in D$
(approximate controllability) and $D$ is a maximal controlled invariant set
with this property. Note that for a control set with nonvoid interior the
control set as well as its interior satisfy the no-return property.

\subsection{Invariance pressure}

Now we introduce the main concepts of the paper, the invariance and outer
invariance pressure generalizing the concepts introduced in Colonius, Cossich
and Santana \cite{Cocosa}.

A pair $(K,Q)$ of nonempty subsets of $M$ is called an \textbf{\textit{admissible
pair}} if $K$ is compact and for each $x\in K$ there exists $\omega
\in\mathcal{U}$ such that $\varphi(\mathbb{R}_{+},x,\omega)\subset Q$ (in
particular, $K\subset Q$).

Given an admissible pair $(K,Q)$ and $\tau>0$, we say that $\mathcal{S}%
\subset\mathcal{U}$ is a $(\tau,K,Q)$-\textbf{\textit{spanning set}}, if for
all $x\in K$ there is $\omega\in\mathcal{S}$ with $\varphi(t,x,\omega)\in Q$
for all $t\in[0,\tau]$. Let $C(U,\mathbb{R})$ denote the set of continuous
functions $f:U\rightarrow\mathbb{R}$.

For $f\in C(U,\mathbb{R})$ denote $(S_{\tau}f)(\omega):=\int_{0}^{\tau
}f(\omega(t))dt$ and
\[
a_{\tau}(f,K,Q):=\inf\left\{  \sum_{\omega\in\mathcal{S}}e^{(S_{\tau}%
f)(\omega)};\ \mathcal{S}\text{ }(\tau,K,Q)\text{-spanning}\right\}  .
\]
The \textbf{\textit{invariance pressure}} $P_{inv}(f,K,Q;\Sigma)$ of control
system $\Sigma$ is given by%
\[
P_{inv}(f,K,Q;\Sigma)=\underset{\tau\rightarrow\infty}{\lim\sup}\frac{1}{\tau
}\log a_{\tau}(f,K,Q).
\]
To simplify the notation we use $P_{inv}(f,K,Q)$ when the considered control
system is clear and, if $K=Q$ we omit the argument $K$ and write $a_{\tau
}(f,Q)$ and $P_{inv}(f,Q)$. Note that, in this case, we assume that $Q$ is
compact and controlled invariant.

Given an admissible pair $(K,Q)$ such that $Q$ is closed in $M$, and a metric
$\varrho$ on $M$, we define the \textbf{\textit{outer invariance pressure}} of
$(K,Q)$ by%

\[
P_{out}(f,K,Q)=P_{out}(f,K,Q;\varrho;\Sigma):=\displaystyle\lim_{\varepsilon
\rightarrow0}P_{inv}(f,K,N_{\varepsilon}(Q)),
\]
where $N_{\varepsilon}(Q)=\{y\in M;\ \exists\ x\in Q\mbox{ with }\varrho
(x,y)<\varepsilon\}$ denotes the $\varepsilon$-neighborhood of $Q$.

Note that the limit for $\varepsilon\rightarrow0$ exists and equals the
supremum over $\varepsilon>0$, since from Proposition \ref{prop4} it follows
that the pairs $(K,N_{\varepsilon}(Q))$ are admissible and that $\varepsilon
_{1}<\varepsilon_{2}$ implies $P_{inv}(f,K,N_{\varepsilon_{1}}(Q))\geq
P_{inv}(f,K,N_{\varepsilon_{2}}(Q))$. Furthermore, $P_{out}(f,K,Q)\leq
P_{inv}(f,K,Q)\leq\infty$ for every admissible pair $(K,Q)$ and $f\in
C(U,\mathbb{R})$.

\section{Properties of invariance pressure\label{Section3}}

In the first part of this section we study several properties of the
invariance pressure and outer invariance pressure. In the second part of this
section we show that conjugations preserve the invariance pressure.

\subsection{Dynamical properties}

We say that two metrics $\varrho_{1}$ and $\varrho_{2}$ on $M$ are uniformly
equivalent on $Q$, if for all $\varepsilon>0$ there exists $\delta>0$ such
that for all $x\in Q$ and for all $y\in M$ with $\varrho_{i}(x,y)<\delta$
implies that $\varrho_{j}(x,y)<\varepsilon$, for $i,j=1,2$, $i\neq j$.

The following proposition states that the value of the outer invariance
pressure of $(K,Q)$ does not change when we consider uniformly equivalent
metrics. Since the proof is similar to Kawan \cite[Proposition 2.1.12]%
{Kawa13}, we will omit it.

\begin{proposition}
Let $(K,Q)$ be an admissible pair such that $Q$ is closed in $M$. If
$\varrho_{1}$ and $\varrho_{2}$ are two metrics on $M$ which are uniformly
equivalent on $Q$, then $P_{out}(f,K,Q;\varrho_{1})=P_{out}(f,K,Q;\varrho
_{2})$ for all $f\in C(U,\mathbb{R})$. If $Q$ is compact, then this is
automatically satisfied, and in this case the outer invariance pressure is
independent of the metric.
\end{proposition}

The next proposition shows that we just need finite spanning sets to get
$a_{\tau}(f,K,Q)$ and it is a reformulation of \cite[Proposition 5]{Cocosa}.

\begin{proposition}
\label{finite} Consider an admissible pair $(K,Q)$ with $Q$ open in $M$ and $f\in C(U,\mathbb{R})$.
Then
\[
a_{\tau}(f,K,Q)=\inf\left\{  \sum_{\omega\in\mathcal{S}}e^{(S_{n}f)(\omega
)};\ \mathcal{S}\text{ is a finite }(\tau,K,Q)\mbox{-spanning set}\right\}  .
\]

\end{proposition}

\begin{proof}
Since $Q$ is open, $\varphi(t,\cdot,\omega)$ is continuous for all $t\in\mathbb{R}$ and
$\omega\in\mathcal{U}$ and $K$ is compact, every $(\tau,K,Q)$-spanning sets
$\mathcal{S}$ admits a finite $(\tau,K,Q)$-spanning subset $\mathcal{S}%
^{\prime}\subset\mathcal{S}$. Now define
\[
\widetilde{a}_{\tau}(f,K,Q):=\inf\left\{  \sum_{\omega\in\mathcal{S}}%
e^{(S_{n}f)(\omega)};\ \mathcal{S}\text{ is a finite }(\tau
,K,Q)\mbox{-spanning set}\right\}  .
\]
Since clearly $a_{\tau}(f,K,Q)\leq\widetilde{a}_{\tau}(f,K,Q)$, we just have
to prove the reverse inequality. Given a $(\tau,K,Q)$-spanning set
$\mathcal{S}$, as shown earlier there is a finite $(\tau,K,Q)$-spanning subset
$\mathcal{S}^{\prime}\subset\mathcal{S}$. Hence $\sum_{\omega\in
\mathcal{S}^{\prime}}e^{(S_{\tau}f)(\omega)}\leq\sum_{\omega\in\mathcal{S}%
}e^{(S_{\tau}f)(\omega)}$, which implies that $\widetilde{a}_{\tau}(f,K,Q)\leq
a_{\tau}(f,K,Q)$.
\end{proof}

The next results of this section show several basic properties that help to
understand the concept of invariance pressure.

\begin{proposition}
\label{prop4} The following assertions hold for an admissible pair $(K,Q)$:

\begin{itemize}
\item[i)] If $0<\tau_{1}<\tau_{2}$ and $f\geq0$, then $a_{\tau_{1}}(f,K,Q)\leq
a_{\tau_{2}}(f,K,Q)$;

\item[ii)] If $Q\subset R$, then $(K,R)$ is admissible and $a_{\tau
}(f,K,Q)\geq a_{\tau}(f,K,R);$ hence $P_{inv}(f,K,Q)\geq P_{inv}(f,K,R);$

\item[iii)] If $L\subset K$ is closed in $M$, then $(L,Q)$ is admissible and
$a_{\tau}(f,L,Q)\leq a_{\tau}(f,K,Q);$ hence $P_{inv}(f,L,Q)\leq
P_{inv}(f,K,Q)$;

\item[iv)] Let $\Sigma^{\prime}:\dot{y}(t)=F^{\prime}(y(t),\omega(t))$ be
another system in $M$ with solution $\varphi^{\prime}$ and a set of admissible
control functions $\mathcal{U}^{\prime}$ containing $\mathcal{U}$ and
$\varphi^{\prime}(t,x,\omega)=\varphi(t,x,\omega)$ whenever $\omega
\in\mathcal{U}$. Then $(K,Q)$ is also admissible for $\Sigma^{\prime}$ and
$P_{inv}(f,K,Q;\Sigma^{\prime})\leq P_{inv}(f,K,Q;\Sigma)$.
\end{itemize}
\end{proposition}

\begin{proposition}
\label{propert} The following assertions hold for an admissible pair $(K,Q)$,
functions $f,g\in C(U,\mathbb{R})$ and $c\in\mathbb{R}$:

\begin{itemize}
\item[i)] $P_{inv}(\mathbf{0},K,Q)=h_{inv}(K,Q)$, where $\mathbf{0}$ is the
null function in $C(U,\mathbb{R})$;

\item[ii)] If $f\leq g$, then $P_{inv}(f,K,Q)\leq P_{inv}(g,K,Q)$. In
particular $h_{inv}(K,Q)+\inf f\leq P_{inv}(f,K,Q)\leq h_{inv}(K,Q)+\sup f$;

\item[iii)] $P_{inv}(f+c,K,Q)=P_{inv}(f,K,Q)+c$;

\item[iv)] $|P_{inv}(f,K,Q)-P_{inv}(g,K,Q)|\leq\Vert f-g\Vert_{\infty}$.
\end{itemize}
\end{proposition}

\begin{proof}
(i) and (ii) are clear from the definition of invariance pressure.

The statements (iii) and (iv) follow analogously to Proposition 13 (ii) of
\cite{Cocosa}.
\end{proof}

The next corollaries deal with the finiteness of invariance pressure.

\begin{corollary}
\label{cor7} Consider $f\in C(U,\mathbb{R})$. 

\begin{itemize}
\item[i)] If $Q$ is open, then $a_{\tau}(f,K,Q)$ is finite for all $\tau>0$;

\item[ii)] If $Q$ is a compact controlled invariant set, then $a_{\tau}(f,Q)$
is either finite for all $\tau>0$ or for none.
\end{itemize}
\end{corollary}

\begin{proof}
The two statements follow from the inequalities
\[
e^{\tau\inf f}r_{inv}(\tau,K,Q)\leq a_{\tau}(f,K,Q)\leq e^{\tau\sup f}%
r_{inv}(\tau,K,Q)
\]
and \cite[Propositions 2.2 and 2.3(i)]{Kawa13}..
\end{proof}

\begin{remark}
Note that Propositions \ref{finite} and \ref{propert} and Corollary
\ref{cor7}(i) also hold for outer invariance pressure.
\end{remark}

As an immediate consequence, we have the following:

\begin{corollary}
If $f\in C(U,\mathbb{R})$ and $Q$ is compact, then the following assertions
are equivalent:

\begin{itemize}
\item[i)] $P_{inv}(f,Q)$ is finite;

\item[ii)] $a_{\tau}(f,Q)$ is finite for some $\tau$;

\item[iii)] $a_{\tau}(f,Q)$ is finite for all $\tau$.
\end{itemize}
\end{corollary}

\begin{proposition}
If $Q$ is a compact controlled invariant set and $f\in C(U,\mathbb{R})$, then
the function $\tau\mapsto a_{\tau}(f,Q)$ is subadditive and therefore
\[
P_{inv}(f,Q)=\lim_{\tau\rightarrow\infty}\frac{1}{\tau}\log a_{\tau}%
(f,Q)=\inf_{\tau>0}\frac{1}{\tau}\log a_{\tau}(f,Q).
\]

\end{proposition}

\begin{proof}
If $a_{\tau}(f,Q)=\infty$ for all $\tau$, the assertion is trivial. Hence, by
Corollary \ref{cor7} (ii) we can assume that $a_{\tau}(f,Q)<\infty$ for all
$\tau$. If we show that $a_{\tau_{1}+\tau_{2}}(f,Q)\leq a_{\tau_{1}}(f,Q)\cdot
a_{\tau_{2}}(f,Q)$ for all $\tau_{1},\tau_{2}>0$, then the result follows from
the subadditivity lemma, see e.g. \cite[Lemma B.7.1]{Kawa13}. To this end,
consider for $j=1,2$ $(\tau_{j},Q)$-spanning sets $\mathcal{S}_{j}$. For
$\omega_{1}\in\mathcal{S}_{1},\omega_{2}\in\mathcal{S}_{2}$ define a control
function $\omega\in\mathcal{U}$ by
\[
\omega(t)=\left\{
\begin{array}
[c]{rcl}%
\omega_{1}(t), & \mbox{if} & t\in\lbrack0,\tau_{1}]\\
\omega_{2}(t-\tau_{1}), & \mbox{if} & t>\tau_{1}%
\end{array}
\right.  .
\]
These functions form a $(\tau_{1}+\tau_{2},Q)$-spanning set. Hence
$a_{\tau_{1}+\tau_{2}}(f,Q)\leq a_{\tau_{1}}(f,Q)\cdot a_{\tau_{2}}(f,Q)$,
which concludes the proof.
\end{proof}

\subsection{Invariance pressure under conjugacy}

Now we show that for systems that are conjugate the respective invariance
pressures coincide.

\begin{definition}
\label{Definition11}Consider two control systems
\[
\Sigma_{1}:\ \dot{x}(t)=F_{1}(x(t),\omega(t))\mbox{ and }\Sigma_{2}:\ \dot
{y}(t)=F_{2}(y(t),\nu(t))
\]
on $M_{1}$ and $M_{2}$, with compact control ranges $U_{1}$ and $U_{2}$, sets
of control functions $\mathcal{U}_{1}$ and $\mathcal{U}_{2}$ and solutions
$\varphi_{1}$ and $\varphi_{2}$, respectively. Let $\pi:\mathbb{R}_{+}\times
M_{1}\rightarrow M_{2}$, $(t,x)\mapsto\pi_{t}(x)$, and $H:U_{1}\rightarrow
U_{2}$ be continuous maps such that the induced map $h_{H}:\mathcal{U}%
_{1}\rightarrow\mathcal{U}_{2}$, $h_{H}(\omega)(t):=H(\omega(t))$ for all
$t\in\mathbb{R}$, satisfies%
\[
\pi_{t}(\varphi_{1}(t,x,\omega))=\varphi_{2}(t,\pi_{0}(x),h_{H}(\omega))\text{
for all }t\in\mathbb{R}_{+},x\in M_{1}\text{ and }\omega\in\mathcal{U}_{1}.
\]
Then $(\pi,H)$ is called a \textbf{\textit{time-variant semi-conjugacy}} from
$\Sigma_{1}$ to $\Sigma_{2}$. If each of the maps $\pi_{t}:M_{1}\rightarrow
M_{2}$ and $H:U_{1}\rightarrow U_{2}$ are homeomorphisms, we call $(\pi,H)$ a
\textbf{\textit{time-variant conjugacy}} from $\Sigma_{1}$ to $\Sigma_{2}$.

Analogously we define a \textbf{\textit{time-invariant \textbf{semi-}%
conjugacy}} and \textbf{\textit{conjugacy}} from $\Sigma_{1}$ to $\Sigma_{2}$
if $\pi$ is independent of $t\in\mathbb{R}_{+}$.
\end{definition}

\begin{proposition}
\label{prop20} Consider two systems as in Definition \ref{Definition11} and
let $(\pi,H)$ be a time-variant semi-conjugacy from $\Sigma_{1}$ to
$\Sigma_{2}$. Further assume that $(K,Q)$ is an admissible pair for
$\Sigma_{1}$ and
\[
\pi_{t}(Q)\subset\pi_{0}(Q)\ \mbox{ for all }t>0.
\]
Then $(\pi_{0}(K),\pi_{0}(Q))$ is an admissible pair for system $\Sigma_{2}$
and
\[
P_{inv}(f\circ H,K,Q;\Sigma_{1})\geq P_{inv}(f,\pi_{0}(K),\pi_{0}%
(Q));\Sigma_{2})
\]
for all $f\in C(U_{2},\mathbb{R})$. Moreover, if $Q$ is compact and the family
$\{\pi_{t}\}_{t\in\mathbb{R}_{+}}$ is pointwise equicontinuous, then
\[
P_{out}(f\circ H,K,Q;\Sigma_{1})\geq P_{out}(f,\pi_{0}(K),\pi_{0}%
(Q));\Sigma_{2})
\]
for all $f\in C(U_{2},\mathbb{R})$.
\end{proposition}

\begin{proof}
In order to show that $(\pi_{0}(K),\pi_{0}(Q))$ is an admissible pair, note
that since $\pi$ is continuous, the set $\pi_{0}(K)$ is compact. Let $y\in
\pi_{0}(K)$, then $y=\pi_{0}(x)$ for some $x\in K$. Since $(K,Q)$ is an
admissible pair, there is $\omega\in\mathcal{U}_{1}$ such that $\varphi
(\mathbb{R}_{+},x,\omega)\subset Q$, and we obtain
\[
\varphi_{2}(t,y,h_{H}(\omega))=\pi_{t}(\varphi_{1}(t,x,\omega))\in\pi
_{t}(Q)\subset\pi_{0}(Q).
\]
Therefore $(\pi_{0}(K),\pi_{0}(Q))$ is an admissible pair for $\Sigma_{2}$.

Now, let $\mathcal{S}\subset\mathcal{U}_{1}$ be a $(\tau,K,Q)$-spanning set.
With the same arguments as above, we find that $h_{H}(\mathcal{S}%
)\subset\mathcal{U}_{2}$ is $(\tau,\pi_{0}(K),\pi_{0}(Q))$-spanning. Hence
\[
\sum_{\mu\in h_{H}(\mathcal{S})}e^{(S_{\tau}f)(\mu)}=\sum_{\omega
\in\mathcal{S}}e^{(S_{\tau}f)(H\circ\omega)}=\sum_{\omega\in\mathcal{S}%
}e^{(S_{\tau}(f\circ H))(\omega)}%
\]
for every $(\tau,K,Q)$-spanning set $\mathcal{S}$, which implies that
\[
a_{\tau}(f,\pi_{0}(K),\pi_{0}(Q))\leq a_{\tau}(f\circ H,K,Q).
\]
Therefore $P_{inv}(f,\pi_{0}(K),\pi_{0}(Q))\leq P_{inv}(f\circ H,K,Q).$

Now assume that $Q$ is compact. Let $\varrho_{1}$ denote a metric on $M_{1}$
and $\varrho_{2}$ a metric on $M_{2}$. By compactness of $Q$, the pointwise
equicontinuity of $\{\pi_{t}\}_{t\in\mathbb{R}_{+}}$ on $Q$ is uniform, hence
for all $\varepsilon>0$, there exists $\delta>0$ such that for all
$t\in\mathbb{R}_{+}$, $x\in Q$ and $y\in M_{1}$ the condition $\varrho
_{1}(x,y)<\delta$ implies $\varrho_{2}(\pi_{t}(x),\pi_{t}(y))<\varepsilon$.

Let $\mathcal{S}\subset\mathcal{U}_{1}$ be a $(\tau,K,N_{\delta}(Q))$-spanning
set with $\delta=\delta(\varepsilon)$ as above. Note that if $y\in\pi_{0}(K)$,
then $y=\pi_{0}(x)$ for some $x\in K$. For $\omega\in\mathcal{S}$ such that
$\varphi_{1}([0,\tau],x,\omega)\subset N_{\delta}(Q)$ and for each
$t\in\lbrack0,\tau]$, there exists $x_{t}\in Q$ with $\varrho_{1}%
(x_{t},\varphi_{1}(t,x,\omega))<\delta$. This implies that for all
$t\in\lbrack0,\tau]$%
\[
\varrho_{2}(\varphi_{2}(t,y,h_{H}(\omega)),\pi_{t}(x_{t}))=\varrho_{2}(\pi
_{t}(\varphi_{1}(t,x,\omega)),\pi_{t}(x_{t}))<\varepsilon.
\]
This shows that $h_{H}(\mathcal{S})\subset\mathcal{U}_{2}$ is a $(\tau,\pi
_{0}(K),N_{\varepsilon}(\pi_{0}(Q)))$-spanning set. We conclude that $a_{\tau
}(f,\pi_{0}(K),N_{\varepsilon}(\pi_{0}(Q)))\leq a_{\tau}(f\circ
H,K,N_{\varepsilon}(Q))$, and hence
\[
P_{out}(f,\pi_{0}(K),\pi_{0}(Q))\leq P_{out}(f\circ H,K,Q).
\]

\end{proof}

\begin{remark}
It is easy to see that if $(\pi,H)$ is a time-variant conjugacy from
$\Sigma_{1}$ to $\Sigma_{2}$, then $(\psi,H^{-1})$ with $\psi_{t}(y):=\pi
_{t}^{-1}(y)$ is a time-variant conjugacy from $\Sigma_{2}$ to $\Sigma_{1}$.
In this case, we have, under the assumptions of the previous proposition,%
\[
P_{inv}(f\circ H,K,Q;\Sigma_{1})=P_{inv}(f,\pi_{0}(K),\pi_{0}(Q));\Sigma
_{2}).
\]
A similar argument holds for time-invariant conjugacies.
\end{remark}

\begin{example}
\label{Example12}Consider two linear control systems $\mathbb{R}^{d}$
\[
\Sigma_{1}:\ \dot{x}(t)=A_{1}x(t)+B_{1}\omega(t)\mbox{ and }\ \Sigma
_{2}:\ \dot{x}(t)=A_{2}x(t)+B_{2}\omega(t),
\]
where $\omega(t)$ is in a compact set $U\subset\mathbb{R}^{m}$ for all
$t\in\mathbb{R}$, $A_{i}\in\mathbb{R}^{d\times d}$ and $B_{i}\in
\mathbb{R}^{d\times m}$ for $i=1,2$. If there is a nonsingular $d\times d$
matrix $T$ such that $A_{2}=TA_{1}T^{-1}$ and $B_{2}=TB_{1}$, then
$(T,id_{U})$ is a time-invariant conjugacy from $\Sigma_{1}$ to $\Sigma_{2}$.
In fact
\begin{align*}
T\left(  \varphi_{1}(t,x,\omega)\right)   &  =T\left(  e^{tA_{1}}x+\int
_{0}^{t}e^{(t-s)A_{1}}B_{1}\omega(s)ds\right) \\
&  =T\left(  e^{tT^{-1}A_{2}T}x+\int_{0}^{t}e^{(t-s)T^{-1}A_{2}T}T^{-1}%
B_{2}\omega(s)ds\right) \\
&  =T\left(  T^{-1}e^{tA_{2}}Tx+\int_{0}^{t}T^{-1}e^{(t-s)A_{2}}TT^{-1}%
B_{2}\omega(s)ds\right) \\
&  =e^{tA_{2}}Tx+\int_{0}^{t}e^{(t-s)A_{2}}B_{2}\omega(s)ds=\varphi
_{2}(t,Tx,h_{id_{U}}(\omega)).
\end{align*}
In this case, it follows for every admissible pair $(K,Q)$, and all $f\in
C(U,\mathbb{R})$%
\[
P_{inv}(f,K,Q;\Sigma_{1})=P_{inv}(f,T(K),T(Q);\Sigma_{2}).
\]

\end{example}

\section{Outer invariance pressure on special sets\label{Section4}}

In this section, we study the outer invariance pressure in isolated sets and
in sets satisfying the no-return property. In the first case, we will see that
the limit for $\varepsilon\rightarrow0$ in the definition of $P_{out}(K,Q)$
becomes superfluous and in the second case we will obtain that the limit
superior in this definition can be replaced by a limit inferior.

We assume that the system $\Sigma$ satisfies the following additional properties:

\begin{itemize}
\item[1)] The set $\mathcal{U}$ of admissible control functions is endowed
with a topology that makes it a sequentially compact space, that is, every
sequence in $\mathcal{U}$ has a convergent subsequence;

\item[2)] The solution map $\varphi:\mathbb{R}_{+}\times M\times
\mathcal{U}\rightarrow M$ is continuous when $\mathcal{U}$ is endowed with the
above topology.
\end{itemize}

These properties are satisfied in particular for a control-affine system of
the form
\[
\dot{x}(t)=X_{0}(x)+\sum_{i=1}^{m}u_{i}X_{i}(x),
\]
where $X_{0},X_{1},\dotsc,X_{m}$ are $C^{1}$ vector fields and $u=(u_{1}%
,\dotsc,u_{m})\in U\subset\mathbb{R}^{m}$ with $U$ compact and convex. Then
the appropriate topology on $\mathcal{U}$ is the weak$^{\ast}$-topology of
$L^{\infty}(\mathbb{R};\mathbb{R}^{m})=L^{1}(\mathbb{R};\mathbb{R}^{m})^{\ast
}$.

A compact set $Q\subset M$ is called \textbf{\textit{isolated}} if there
exists $\delta_{0}>0$ such that for all $(x,\omega)\in\overline{N_{\delta_{0}%
}(Q)}\times\mathcal{U}$ the following implication holds:
\begin{equation}
\varphi(\mathbb{R}_{+},x,\omega)\subset\overline{N_{\delta_{0}}(Q)}%
\Rightarrow\varphi(\mathbb{R}_{+},x,\omega)\subset Q. \label{isolated}%
\end{equation}

\begin{proposition}
Let $(K,Q)$ be an admissible pair such that $Q$ is compact and isolated with
constant $\delta_{0}$. Then it holds, for all $f\in C(U,\mathbb{R})$
\[
P_{out}(f,K,Q)=P_{inv}(f,K,N_{\varepsilon}(Q))\text{ for all }\varepsilon
\in(0,\delta_{0}],
\]

\end{proposition}

\begin{proof}
Since $M$ is locally compact, by Kawan \cite[ Lemma A.4.2]{Kawa13} we may
assume that $\delta_{0}$ is small enough that $\overline{N_{\delta_{0}}(Q)}$
is compact, since assumption (\ref{isolated}) is also satisfied for smaller
$\delta_{0}$.

By an argument similar to \cite[Proposition 2.2.17]{Kawa13}, we can see that
for all $\rho>0$ and for all $\varepsilon\in(0,\delta_{0}]$ there is
$n\in\mathbb{N}$ such that for all $(x,\omega)\in\overline{N_{\delta_{0}}%
(Q)}\times\mathcal{U}$ $\max_{t\in\lbrack0,n]}\mbox{dist}(\varphi
(t,x,\omega),Q)\leq\varepsilon$ implies $\mbox{dist}(x,Q)<\rho$.

Now let $0<\varepsilon_{1}<\varepsilon_{2}\leq\delta_{0}$. Then there exists
$n\in\mathbb{N}$ such that for all $(x,\omega)\in\overline{N_{\delta_{0}}%
(Q)}\times\mathcal{U}$ it holds that $\max_{t\in\lbrack0,n]}%
\mbox{dist}(\varphi(t,x,\omega),Q)\leq\varepsilon_{2}$ implies
$\mbox{dist}(x,Q)<\varepsilon_{1}$. For arbitrary $\tau>0$, let $\mathcal{S}$
be a $(n+\tau,K,N_{\varepsilon_{2}}(Q))$-spanning set. For $x\in K$, there
exists $\omega_{x}\in\mathcal{S}$ with $\varphi([0,n+\tau],x,\omega
_{x})\subset N_{\varepsilon_{2}}(Q).$ For every $s\in\lbrack0,\tau]$, we
obtain $\max_{t\in\lbrack0,n]}\mbox{dist}(\varphi(t,\varphi(s,x,\omega
_{x}),\Theta_{s}\omega_{x}),Q)=\max_{t\in\lbrack0,n]}\mbox{dist}(\varphi
(t+s,x,\omega_{x}),Q)<\varepsilon_{2}$. Hence we have $\mbox{dist}(\varphi
(s,x,\omega_{x}),Q)<\varepsilon_{1}$ for all$\ s\in\lbrack0,\tau]$, which
implies that $\mathcal{S}$ is a $(\tau,K,N_{\varepsilon_{1}}(Q))$-spanning
set. Therefore, given $g\in C(U,\mathbb{R})$, $g\geq0$, we get
\[
a_{\tau}(g,K,N_{\varepsilon_{1}}(Q))\leq a_{n+\tau}(g,K,N_{\varepsilon_{2}%
}(Q)),\ \forall\tau>0,
\]
which implies $P_{inv}(f,K,N_{\varepsilon_{1}}(Q))\leq P_{inv}%
(f,K,N_{\varepsilon_{2}}(Q))$, for all $f\in C(U,\mathbb{R})$.

By Proposition \ref{prop4} (ii) we have $P_{inv}(f,K,N_{\varepsilon_{2}%
}(Q))\leq P_{inv}(f,K,N_{\varepsilon_{1}}(Q))$ and the proof is complete.
\end{proof}

\begin{proposition}
Let $Q\subset M$ be a set with the no-return property. Assume that $(K_{1},Q)$
and $(K_{2},Q)$ are two admissible pairs such that $K_{2}$ has a nonempty
interior and
\[
\forall x\in K_{1}~\exists\omega_{x}\in\mathcal{U}~\exists\tau_{x}%
>0:\varphi(\tau_{x},x,\omega_{x})\in\mathrm{int}K.
\]
Then for all $f\in C(U,\mathbb{R})$%
\[
P_{inv}(f,K_{1},Q)\leq P_{inv}(f,K_{2},Q).
\]

\end{proposition}

\begin{proof}
Note that if there exists $\tau_{0}$ such that $a_{\tau}(f,K_{2},Q)=+\infty$
for all $\tau\geq\tau_{0}$, then $P_{inv}(f,K_{2},Q)=+\infty$ and hence the
assertion holds.

If this is not the case, we can get a sequence $\tau_{k}\rightarrow\infty$
such that $a_{\tau_{k}}(f,K_{2},Q)$ is finite for all $k$. For all $x\in
K_{1}$, let $\omega_{x}\in\mathcal{U}$ and $\tau_{x}>0$ as in the assumption.
Since $\varphi(\tau_{x},\cdot,\omega_{x})$ is continuous, we find, for every
$x\in K_{1}$, an open neighborhood $V_{x}$ of $x$ such that $\varphi(\tau
_{x},V_{x},\omega_{x})\subset\mathrm{int}K_{2}$. By the no-return property of
$Q$, we have $\varphi([0,\tau_{x}],y,\omega_{x})\subset Q$, for all $y\in
K_{1}\cap V_{x}$. The family $\{V_{x}\}_{x\in K_{1}}$ is an open cover of
$K_{1}$ and by compactness there exist $x_{1},\dotsc,x_{n}\in K_{1}$ with
$K_{1}\subset\cup_{i=1}^{n}V_{x_{i}}$. Now, let $\mathcal{S}:=\{\mu_{1}%
,\dotsc,\mu_{k}\}$ be a finite $(\tau,K_{2},Q)$-spanning set, for some
$\tau>\tau_{M}-\tau_{m}$, where $\tau_{M}:=\max_{1\leq i\leq n}\tau_{x_{i}}$
and $\tau_{m}:=\min_{1\leq i\leq n}\tau_{x_{i}}$.

For every index pair $(i,j)$ with $1\leq i\leq n$, $1\leq j\leq k$ such that
there exists $x\in K_{1}$ with $y_{x}:=\varphi(\tau_{x_{i}},x,\omega_{x_{i}%
})\in\mathrm{int}K_{2}$ and $\varphi([0,\tau],y_{x},\mu_{j})\subset Q$, we can
define a control function
\[
\nu_{ij}(t)=\left\{
\begin{array}
[c]{rcl}%
\omega_{x_{i}}(t), & \mbox{if} & t\in\lbrack0,\tau_{x_{i}}]\\
\mu_{j}(t-\tau_{x_{i}}), & \mbox{if} & t>\tau_{x_{i}}%
\end{array}
\right.  .
\]
Define the set $\widetilde{\mathcal{S}}$ of all these control functions. Let
$\widetilde{\tau}:=\tau+\tau_{m}$, hence $\tau\geq\widetilde{\tau}-\tau_{M}$.
Then $\widetilde{\mathcal{S}}$ is a $(\widetilde{\tau},K_{1},Q)$-spanning set
by construction, and consequently, for all $f\in C(U,\mathbb{R})$, $f\geq0$,
we have
\begin{align*}
(S_{\widetilde{\tau}}f)(\nu_{ij})  &  =(S_{\tau_{x_{i}}}f)(\omega_{x_{i}%
})+\int_{\tau_{x_{i}}}^{\widetilde{\tau}}f(\mu_{j}(t-\tau_{x_{i}}))dt\\
&  =(S_{\tau_{x_{i}}}f)(\omega_{x_{i}})+\int_{0}^{\widetilde{\tau}-\tau
_{x_{i}}}f(\mu_{j}(t))dt\\
&  =(S_{\tau_{x_{i}}}f)(\omega_{x_{i}})+(S_{\widetilde{\tau}-\tau_{x_{i}}%
}f)(\mu_{j})\\
&  \leq(S_{\tau_{x_{i}}}f)(\omega_{x_{i}})+(S_{\tau}f)(\mu_{j}).
\end{align*}
Hence
\begin{align*}
a_{\tau}(f,K_{1},Q)  &  \leq a_{\widetilde{\tau}}(f,K_{1},Q)\leq\sum_{\nu
_{ij}\in\widetilde{\mathcal{S}}}e^{(S_{\widetilde{\tau}}f)(\nu_{ij})}\leq
\sum_{1\leq i\leq n,\ \mu\in\mathcal{S}}e^{(S_{\tau_{x_{i}}}f)(\omega_{x_{i}%
})}e^{(S_{\tau}f)(\mu)}\\
&  \leq\sum_{1\leq i\leq n}e^{(S_{\tau_{x_{i}}}f)(\omega_{x_{i}})}\cdot
\sum_{\mu\in\mathcal{S}}e^{(S_{\tau}f))(\mu)}\leq ne^{\Vert f\Vert\tau_{M}%
}\sum_{\mu\in\mathcal{S}}e^{(S_{\tau}f(\mu)},
\end{align*}
because $0\leq\widetilde{\tau}-\tau_{x_{i}}\leq\tau$. Since this inequality
holds for all finite $(\tau,K_{2},Q)$-spanning sets, we have
\[
a_{\tau}(f,K_{1},Q)\leq ne^{\Vert f\Vert\tau_{M}}a_{\tau}(f,K_{2}%
,Q),\ \tau>\tau_{M}-\tau_{m}.
\]
Therefore, we obtain for all $f\in C(U,\mathbb{R})$,$f\geq0$,%
\[
P_{inv}(f,K_{1},Q)\leq P_{inv}(f,K_{2},Q).
\]
Now consider an arbitrary $f\in C(U,\mathbb{R})$. Then $\tilde{f}\in
C(U,\mathbb{R})$ given by $\tilde{f}(u)=f(u)-\inf f$ satisfies $\tilde{f}%
\geq0$. Using Proposition \ref{propert} (iii) it follows that
\begin{align*}
P_{inv}(f,K_{1},Q)  &  =P_{inv}(\tilde{f},K_{1},Q)+\inf_{u\in U}f(u)\leq
P_{inv}(\tilde{f},K_{2},Q)+\inf_{u\in U}f(u)\\
&  =P_{inv}(f,K_{2},Q)-\inf_{u\in U}f(u)+\inf_{u\in U}f(u)\\
&  =P_{inv}(f,K_{2},Q).
\end{align*}

\end{proof}

\begin{corollary}
\label{cor1} Let $D\subset M$ be a control set and let $K_{1},K_{2}\subset D$
be two compact subsets with nonempty interior. Then $(K_{1},D)$ and
$(K_{2},D)$ are admissible pairs and for all $f\in C(U,\mathbb{R})$ we have
\[
P_{inv}(f,K_{1},D)=P_{inv}(f,K_{2},D).
\]

\end{corollary}

\begin{proof}
This follows, since control sets with nonvoid interior satisfy the no-return property.
\end{proof}

\section{Outer invariance pressure for inner control sets\label{Section5}}

In this section, we will show that the limit superior in the definition of
invariance pressure of a control set can be replaced by the limit inferior, if
certain controllability properties near the control set are satisfied.

A control set $D$ $\subset M$ is called an \textbf{\textit{inner control set}}
if there exists an increasing family of compact and convex sets $\{U_{\rho
}\}_{\rho\in\lbrack0,1]}$ in $\mathbb{R}^{m}$ (i.e., $U_{\rho_{1}}\subset
U_{\rho_{2}}$ for $\rho_{1}<\rho_{2}$), such that for every $\rho\in
\lbrack0,1]$ the system $\Sigma$ with control range $U_{\rho}$ (instead of
$U$) has a control set $D_{\rho}$ with nonvoid interior and compact closure,
and the following conditions are satisfied:

\begin{itemize}
\item[i)] $U=U_{0}$ and $D=D_{0}$;

\item[ii)] $\overline{D_{\rho_{1}}}\subset\mathrm{int}(D_{\rho_{2}})$ whenever
$\rho_{1}<\rho_{2}$;

\item[iii)] For every neighborhood $W$ of $\overline{D}$ there is $\rho
\in\lbrack0,1)$ with $\overline{D_{\rho}}\subset W$.
\end{itemize}

This notion (slightly modified) is taken from Kawan \cite[Definition
2.6]{Kawa13}. Below, we will consider an inner control set $D=D_{0}$
(corresponding to the control range $U=U_{0}$) and characterize the invariance
pressure of the controlled invariant set $Q=\overline{D}$ with respect to the
larger control range $U_{1}\supset U_{0}$.

The following result shows that for admissible pairs $(K,Q)$ where $Q$ is the
closure of an inner control set, the limit superior in the definition of outer
invariance pressure can be replaced by the limit inferior. The proof follows
\cite[Proposition 2.16]{Kawa13}.

\begin{proposition}
Let $Q$ be the closure of an inner control set $D$ of a system $\Sigma$. Then
for every compact set $K\subset D$, the pair $(K,Q)$ is admissible for the
system with control range $U_{1}$ and if $\mathrm{int}K\neq\emptyset$ we have
\[
P_{out}(f,K,Q)=\lim_{\varepsilon\searrow0}\liminf_{\tau\rightarrow\infty}%
\frac{1}{\tau}\log a_{\tau}(f,K,N_{\varepsilon}(Q))\text{ for every }f\in
C(U,\mathbb{R}).
\]

\end{proposition}

\begin{proof}
First observe that (by the Tietze extension theorem) every continuous function
$f\in C(U,\mathbb{R})$ can be extended to a continuous function $f\in
C(U_{1},\mathbb{R})$. We fix such an extension. Our proof will show that
$P_{out}(f,K,Q)$ does not depend on this extension.

From conditions (ii) and (iii) of inner control sets, it follows that exists a
monotonically decreasing sequence $(\rho_{n})_{n\in\mathbb{N}}$ in $[0,1)$
with $D_{\rho_{n}}\subset N_{1/n}(Q)$ for all $n\in\mathbb{N}$. Since
$Q=\overline{D}\subset\mathrm{int}(D_{\rho_{n}})$ for all $n\in\mathbb{N}$, we
can find a monotonically decreasing sequence $(\varepsilon_{n})_{n\in
\mathbb{N}}$ of positive real numbers with $\varepsilon_{n}\searrow0$ such
that $\overline{N_{\varepsilon_{n}}(Q)}\subset D_{\rho_{n}}$ for all
$n\in\mathbb{N}$. For each $n\in\mathbb{N}$ it is possible to steer all points
of $N_{\varepsilon_{n}}(Q)$ to $K$ with finitely many control functions using
the control range $U_{\rho_{n}}$. In fact, since $\overline{N_{\varepsilon
_{n}}(Q)}$ and $K$ are subsets of the control set $D_{\rho_{n}}$ for each $n$,
then for all $x\in\overline{N_{\varepsilon_{n}}(Q)}$, there exist $t_{x}%
^{n}>0$ and $\mu_{x}^{n}\in\mathcal{U}$, $\mu_{x}^{n}(t)\in U_{\rho_{n}}$ for
all $t$, such that $\varphi(t_{x}^{n},x,\mu_{x}^{n})\in\mathrm{int}K$ by the
approximate controllability of the control set $D_{\rho_{n}}$. Continuity
implies that there exists a neighborhood $W_{x}^{n}$ of $x$ such that
$\varphi(t_{x}^{n},W_{x}^{n},\mu_{x}^{n})\subset\mathrm{int}K$. By compactness
there exist $x_{1}^{n},\dotsc,x_{k_{n}}^{n}\in\overline{N_{\varepsilon_{n}%
}(Q)}$ such that%
\[
\overline{N_{\varepsilon_{n}}(Q)}\subset\bigcup_{i=1}^{k_{n}}W_{x_{i}}^{n}.
\]
Denote $\mathcal{S}_{n}:=\{\mu_{1}^{n},\dotsc,\mu_{k_{n}}^{n}\}$, where
$\mu_{j}^{n}=\mu_{x_{j}}^{n}$, and $\tau_{j}^{n}:=t_{x_{j}}^{n}$. Observe that
given $x\in N_{\varepsilon_{n}}(Q)$, the trajectory $\varphi(t,x,\mu_{j}^{n}%
)$, $t\in\lbrack0,\tau_{j}^{n}]$, does not leave the control set $D_{\rho_{n}%
}\subset N_{1/n}(Q)$ by the no-return property.

For every $\tau>\tau_{M}^{n}:=\max\{\tau_{j}^{n};j=1,\dotsc k_{n}\}$ consider a
finite $(\tau,K,N_{\varepsilon}(Q))$-spanning set $\mathcal{S}=\{\omega
_{1},\dotsc,\omega_{k}\}$, where $\varepsilon\in(0,\varepsilon_{n}]$ and the
controls take values in $U_{0}$. Let $\widetilde{\mathcal{S}}$ be the set
consisting of the functions
\[
\nu_{ij}^{n}(t)=\left\{
\begin{array}
[c]{rcl}%
\omega_{i}(t), & \mbox{if} & t\in\lbrack0,\tau-\tau_{j}^{n}]\\
\mu_{j}^{n}(t-\tau_{j}^{n}), & \mbox{if} & t\in(\tau-\tau_{j}^{n},\tau]
\end{array}
\right.  \text{, }1\leq i\leq k\text{ and }1\leq j\leq k_{n}.
\]
Thus for every $x\in K$ there is a control in $\tilde{S}$ keeping the
corresponding trajectory in $N_{\varepsilon}(Q)$ up to time $\tau-\tau_{j}%
^{n}$ and then steering the system back to $K$. Now, for $m\in\mathbb{N}$,
define $\widehat{\mathcal{S}}$ as the set obtained by $m$ iterations of the
elements of $\widetilde{\mathcal{S}}$. Hence $\widehat{\mathcal{S}}$ is a
$(m\tau,K,N_{1/n}(Q))$-spanning set with $\#\widehat{\mathcal{S}}\leq\left(
\#\mathcal{S}\right)  ^{m}\cdot\left(  \#\mathcal{S}_{n}\right)  ^{m}<\infty$.

We compute for $\nu\in\widehat{\mathcal{S}}$
\begin{align*}
(S_{m\tau}f)(\nu)  &  =\int_{0}^{m\tau}f(\nu(t))dt=\int_{0}^{\tau}f(\nu
_{i_{1},j_{1}}(t))dt+\cdots+\int_{(m-1)\tau}^{m\tau}f(\nu_{i_{m},j_{m}%
}(t))dt\\
&  =\int_{0}^{\tau-\tau_{j_{1}}}f(\omega_{i_{1}}(t))dt+\int_{\tau-\tau_{j_{1}%
}}^{\tau}f(\mu_{j_{1}}^{n}(t-\tau_{j_{1}}^{n}))dt+\cdots+\\
&  ~\quad+\int_{(m-1)\tau}^{m\tau-\tau_{j_{m}}}f(\omega_{i_{m}}(t-(m-1)\tau
))dt+\int_{m\tau-\tau_{j_{m}}}^{m\tau}f(\mu_{j_{m}}^{n}(t-(m\tau-\tau_{j_{m}%
}^{n})))dt\\
&  =\int_{0}^{\tau-\tau_{j_{1}}}f(\omega_{i_{1}}(t))dt+\int_{0}^{\tau_{j_{1}}%
}f(\mu_{j_{1}}^{n}(t))dt+\cdots+\\
&  ~\quad+\int_{0}^{\tau-\tau_{j_{m}}}f(\omega_{i_{m}}(t))dt+\int_{0}%
^{\tau_{j_{m}}}f(\mu_{j_{m}}^{n}(t))dt\\
&  \leq(S_{\tau}f)(\omega_{i_{1}})+\cdots+(S_{\tau}f)(\omega_{i_{m}}%
)+2m\tau_{M}^{n}\sup f.
\end{align*}
This implies for all $(\tau,K,N_{\varepsilon}(Q))$-spanning sets $\mathcal{S}$
and $\varepsilon\in(0,\varepsilon_{n}]$%
\begin{align*}
a_{m\tau}(f,K,N_{1/n}(Q))  &  \leq\sum_{\nu\in\widehat{\mathcal{S}}%
}e^{(S_{m\tau}f)(\nu)}\\
&  \leq e^{2m\tau_{M}^{n}\sup f}\cdot\sum_{\omega_{i_{l}}\in\mathcal{S}%
;\ 1\leq l\leq m}e^{(S_{\tau}f)(\omega_{i_{1}})+\cdots+(S_{\tau}%
f)(\omega_{i_{m}})}\\
&  \leq e^{2m\tau_{M}^{n}\sup f}\cdot\left(  \sum_{\omega\in\mathcal{S}%
}e^{(S_{\tau}f)(\omega)}\right)  \cdots\left(  \sum_{\omega\in\mathcal{S}%
}e^{(S_{\tau}f)(\omega)}\right) \\
&  =e^{2m\tau_{M}^{n}\sup f}\cdot\left(  \sum_{\omega\in\mathcal{S}%
}e^{(S_{\tau}f)(\omega)}\right)  ^{m}.
\end{align*}
It follows that $a_{m\tau}(f,K,N_{1/n}(Q))\leq e^{2m\tau_{M}^{n}\sup f}%
\cdot(a_{\tau}(f,K,N_{\varepsilon}(Q)))^{m}$ for all $m\in\mathbb{N}$,
$\tau>0$ and $\varepsilon\in(0,\varepsilon_{n}]$. By discretization of time we
get
\begin{align*}
P_{inv}(f,K,N_{1/n}(Q))  &  =\limsup_{m\rightarrow\infty}\frac{1}{m\tau}\log
a_{m\tau}(f,K,N_{1/n}(Q))\\
&  \leq\limsup_{m\rightarrow\infty}\frac{1}{m\tau}(2m\tau_{M}^{n}\sup f+m\log
a_{\tau}(f,K,N_{\varepsilon}(Q)))\\
&  =\frac{2}{\tau}\tau_{M}^{n}\sup f+\frac{1}{\tau}\log a_{\tau}%
(f,K,N_{\varepsilon}(Q)).
\end{align*}
Therefore we obtain
\begin{align*}
&  P_{inv}(f,K,N_{1/n}(Q))\\
&  \leq\lim_{\varepsilon\searrow0}\liminf_{\tau\rightarrow\infty}\left(
\frac{2}{\tau}\tau_{M}^{n}\sup f+\limsup_{\tau\rightarrow\infty}\frac{1}{\tau
}\log a_{\tau}(f,K,N_{\varepsilon}(Q))\right) \\
&  =\lim_{\varepsilon\searrow0}\liminf_{\tau\rightarrow\infty}\frac{1}{\tau
}\log a_{\tau}(f,K,N_{\varepsilon}(Q)).
\end{align*}
Since this inequality holds for every $n\in\mathbb{N}$, the assertion follows.
\end{proof}

\begin{remark}
Note that it does not necessarily follow that the limit%
\[
\lim_{\tau\rightarrow\infty}\log a_{\tau}(f,K,N_{\varepsilon}(Q))
\]
exists for any $\varepsilon>0$.
\end{remark}

\section{Invariance pressure for linear control systems\label{Section6}}

In this section, we prove a main result of this paper. We consider a class of
linear control systems given by $(A,B)$ where $A$ is hyperbolic (that is, $A$
has no eigenvalues on the imaginary axis). The control range $U\subset
\mathbb{R}^{m}$ is a compact neighborhood of the origin, and we suppose that
the pair $(A,B)$ is controllable (that is, rank$\left[  B\ AB\ \cdots
\ A^{d-1}B\right]  =d$). Consequently, the system is locally accessible.

From Hinrichsen and Pritchard \cite[Theorems 6.2.22 and 6.2.23]{HiP18} (cf.
also Colonius and Kliemann \cite[Example 3.2.16]{CoKli}) we get the following
result on existence and uniqueness of a control set.

\begin{theorem}
\label{Theorem_unique}Consider a linear control system of the form
$\Sigma_{lin}$ and assume that the pair $(A,B)$ is controllable and the
control range $U$ is a compact neighborhood of the origin.

(i) Then there is a unique control set $D$ with nonempty interior, it is
convex and satisfies%
\[
0\in\mathrm{int}D\text{ and }D=\mathcal{O}^{-}(x)\cap\overline{\mathcal{O}%
^{+}(x)}\text{ for every }x\in\mathrm{int}D.
\]

(ii) $D$ is closed if and only if $\mathcal{O}^{+}(x)\subset D$ for all $x\in
D$.

(iii) The control set $D$ is bounded if and only if $A$ is hyperbolic.
\end{theorem}

The following result generalizes and improves \cite[Theorem 27]{Cocosa} (where
the outer invariance pressure was considered). The proof follows Kawan
\cite[Theorem 4.3]{Kawa11}, \cite[Theorem 5.1]{Kawa13}, considerably
simplified for the linear situation.

\begin{theorem}
\label{linear_main}Consider a linear control system of the form $\Sigma_{lin}$
and assume that the pair $(A,B)$ is controllable, that $A$ is hyperbolic and
the control range $U$ is a compact neighborhood of the origin in
$\mathbb{R}^{m}$. Let $D$ be the unique control set with nonempty interior and
let $f\in C(U,\mathbb{R})$.

Then for every compact set $K\subset D$ the pair $(K,D)$ is admissible and
\[
P_{inv}(f,K,D)\leq\sum_{\lambda\in\sigma(A)}\max\{0,n_{\lambda}%
\mathrm{\operatorname{Re}}(\lambda)\}+\inf\frac{1}{T}\int_{0}^{T}%
f(\omega(s))ds,
\]
where the infimum is taken over all $T>0$ and all $T$-periodic controls
$\omega(\cdot)$ with a $T$-periodic trajectory $x(\cdot)$ in $\mathrm{int}D$
such that $\{\omega(t);t\in\lbrack0,T]\}$ is contained in a compact subset of
$\mathrm{int}U$.
\end{theorem}

\begin{proof}
We will construct a compact subset $K\subset D$ with nonvoid interior such
that
\[
P_{inv}(f,K,D)\leq\sum_{\lambda\in\sigma(A)}\max\{0,n_{\lambda}%
\mathrm{\operatorname{Re}}(\lambda)\}+\inf\frac{1}{T}\int_{0}^{T}f(\omega
_{0}(s))ds.
\]
Then the assertion will follow, since every compact subset of $D$ is contained
in a compact subset $K$ of $D$ with nonvoid interior and the invariance
pressure is independent of the choice of such a set $K$ by Corollary
\ref{cor1}$.$

For the proof consider a $\tau_{0}$-periodic control $\omega_{0}(\cdot)$ with
$\tau_{0}$-periodic trajectory as in the statement of the theorem. We can
transform $A$ into real Jordan form $R$ without changing the invariance
pressure, cf. Example \ref{Example12}, and%
\begin{equation}
x_{0}=e^{R\tau_{0}}x_{0}+\int_{0}^{\tau_{0}}e^{R(\tau_{0}-s)}B\omega_{0}(s)ds.
\label{periodic}%
\end{equation}

\textbf{Step 1:} Choose a basis $B$ of $\mathbb{R}^{d}$ adapted to the real
Jordan structure of $R$ and let $L_{1}(R),\dotsc,L_{r}(R)$ be the different
Lyapunov spaces of $R$, that is, the sums of the generalized eigenspaces
corresponding to eigenvalues with the same real part $\rho_{j}$. Then we have
the decomposition%
\[
\mathbb{R}^{d}=L_{1}(R)\oplus\cdots\oplus L_{r}(R).
\]
Let $d_{j}=\dim L_{j}(R)$ and denote the restriction of $R$ to $L_{j}(R)$ by
$R_{j}$. Now take an inner product on $\mathbb{R}^{d}$ such that the basis $B$
is orthonormal with respect to this inner product and let $\left\Vert
\cdot\right\Vert $ denote the induced norm.

\textbf{Step 2:} We fix some constants: Let $S_{0}$ be a real number which
satisfies%
\[
S_{0}>\sum\limits_{j=1}^{r}\max(0,d_{j}\rho_{j})
\]
and choose $\xi=\xi(S_{0})>0$ such that%
\[
0<d\xi<S_{0}-\sum\limits_{j=1}^{r}\max(0,d_{j}\rho_{j}).
\]
Let $\delta\in(0,\xi)$ be chosen small enough such that $\rho_{j}<0$ implies
$\rho_{j}+\delta<0$ for all $j$. It follows that there exists a constant
$c=c(\delta)\geq1$ such that for all $j\ $and for all $k\in\mathbb{N}$%
\[
\left\Vert e^{tR_{j}}\right\Vert \leq ce^{(\rho_{j}+\delta)t}\text{ for all
}t\geq0.
\]
For every $t>0$ we define positive integers%
\[
M_{j}(t)=\left\{
\begin{array}
[c]{ccc}%
\left\lfloor e^{(\rho_{j}+\xi)t}\right\rfloor +1 & \text{if} & \rho_{j}\geq0\\
1 & \text{if} & \rho_{j}<0
\end{array}
\right.  .
\]
Moreover, we define a function $\beta:(0,\infty)\rightarrow(0,\infty)$ by%
\[
\beta(t)=\max_{1\leq j\leq r}\left[  e^{(\rho_{j}+\delta)t}\frac{\sqrt{d_{j}}%
}{M_{j}(t)}\right]  ,t>0.
\]
If $\rho_{j}<0$, then $\rho_{j}+\delta<0$ and $M_{j}(t)\equiv1$. This implies
that $e^{(\rho_{j}+\delta)t}/M_{j}(t)$ converges to zero for $t\rightarrow
\infty$. If $\rho_j\geq0$, we have $M_{j}(t)\geq e^{(\rho_{j}+\xi)t}$ and hence
\begin{equation}
e^{(\rho_{j}+\delta)t}\frac{\sqrt{d_{j}}}{M_{j}(t)}\leq e^{(\rho_{j}+\delta
)t}\frac{\sqrt{d_{j}}}{e^{(\rho_{j}+\xi)t}}=e^{(\delta-\xi)t}\sqrt{d_{j}}.
\label{beta2}%
\end{equation}
Since $\delta\in(0,\xi)$, we have $\delta-\xi<0$ and hence the terms above
converge to zero for $t\rightarrow\infty$. Thus, also $\beta(t)\rightarrow0$
for $t\rightarrow\infty$. Since we assume controllability of $(A,B)$ there
exists $C>0$ such that for every $\lambda\in\mathbb{R}^{d}$ there is a control
$\omega\in L^{\infty}(0,\tau,\mathbb{R}^{m})$ with%
\begin{equation}
\varphi(\tau_{0},\lambda,\omega)=e^{R\tau_{0}}\lambda+\int_{0}^{\tau_{0}%
}e^{R(\tau_{0}-s)}B\omega(s)ds=0\text{ and }\left\Vert \omega\right\Vert
_{\infty}\leq C\left\Vert \lambda\right\Vert . \label{Kawan5.9}%
\end{equation}
The inequality follows by the inverse mapping theorem.

For $b_{0}>0$ let $\mathcal{C}$ be the $d$-dimensional compact cube
$\mathcal{C}$ in $\mathbb{R}^{d}$ centered at the origin with sides of length
$2b_{0}$ parallel to the vectors of the basis $B$. Choose $b_{0}$ small enough
such that, with $x_{0}:=x(0)$
\[
K:=x_{0}+\mathcal{C}\subset D
\]
and $\overline{B(\omega_{0}(t),Cb_{0})}\subset U$ for almost all $t\in
\lbrack0,\tau_{0}]$. This is possible, since $x_{0}\in\mathrm{int}D$ and
almost all values $\omega_{0}(t)$ are in a compact subset of the interior of
$U$.

\textbf{Step 3.} Let $\varepsilon>0$ and $\tau=k\tau_{0}$ with $k\in
\mathbb{N}$. We may take $k\in\mathbb{N}$ large enough such that%
\begin{equation}
\frac{d}{\tau}\log2<\varepsilon. \label{two}%
\end{equation}
Furthermore, we may choose $b_{0}$ small enough such that $Cb_{0}<\varepsilon
$. Partition $\mathcal{C}$ by dividing each coordinate axis corresponding to a
component of the $j$th Lyapunov space $L_{j}(R)$ into $M_{j}(\tau)$ intervals
of equal length. The total number of subcuboids in this partition of
$\mathcal{C}$ is $\prod_{j=1}^{r}M_{j}(\tau)^{d_{j}}$.

Next we will show that it suffices to take $\prod_{j=1}^{r}M_{j}(\tau)^{d_{j}%
}$ control functions to steer the system from all states in $x_{0}%
+\mathcal{C}$ back to $x_{0}+\mathcal{C}$ in time $\tau$ such that the
controls are within distance $\varepsilon$ to $\omega_{0}$.

Let $\lambda$ be the center of a subcuboid. By (\ref{Kawan5.9}) there exists
$\omega\in L^{\infty}(0,\tau,\mathbb{R}^{m})$ such that%
\[
\varphi(\tau,\lambda,\omega)=0\text{ and }\left\Vert \omega\right\Vert _{\infty
}\leq C\left\Vert \lambda\right\Vert \leq Cb_{0}<\varepsilon.
\]
Hence $\omega(t)\in U$ for a.a. $t\in\lbrack0,\tau]$ and, using
(\ref{periodic}) and linearity, we find that $x_{0}+\lambda$ is steered by
$\omega_{0}+\omega$ in time $\tau=k\tau_{0}$ to $x_{0}$,%
\begin{equation}
\varphi(\tau,x_{0}+\lambda,\omega_{0}+\omega)=\varphi(\tau,x_{0},\omega_{0}%
)+\varphi(\tau,\lambda,\omega)=x_{0}. \label{periodic2}%
\end{equation}
Now consider an arbitrary point $x\in\mathcal{C}$. Then it lies in one of the
subcuboids and we denote the corresponding center of this subcuboid by
$\lambda$ with associated control $\omega$. We will show that $\omega
_{0}+\omega$ also steers $x_{0}+x$ back to $x_{0}+\mathcal{C}$. Observe that%
\[
\left\Vert x-\lambda\right\Vert \leq\frac{b_{0}}{M_{j}(\tau)}\sqrt{d_{j}}.
\]
This implies that%
\[
\left\Vert e^{\tau R}x-e^{\tau R}\lambda\right\Vert \leq\left\Vert
e^{(k\tau_{0}R_{j})}\right\Vert \left\Vert x-\lambda\right\Vert \leq
ce^{(\rho_{j}+\delta)k\tau_{0}}\frac{b_{0}}{M_{j}(k\tau_{0})}\sqrt{d_{j}%
}\rightarrow0\text{ for }k\rightarrow\infty,
\]
and hence for $k$ large enough $\left\Vert e^{\tau R}x-e^{\tau R}%
\lambda\right\Vert \leq b_{0}$. This implies that the solution%
\[
\varphi(t,x_{0}+x,\omega_{0}+\omega)=e^{tR}(x_{0}+x)+\int_{0}^{t}e^{R(t-s)}%
B\left[  \omega_{0}(s)+\omega(s)\right]  ds,t\geq0,
\]
satisfies for $k$ large enough by (\ref{periodic2}) and linearity,%
\begin{align*}
&  \left\Vert \varphi(\tau,x_{0}+x,\omega_{0}+\omega)-x_{0}\right\Vert \\
&  =\left\Vert e^{\tau R}(x_{0}+x)+\int_{0}^{\tau}e^{R(\tau-s)}B\left[
\omega_{0}(s)+\omega(s)\right]  ds-x_{0}\right\Vert \\
&  \leq\left\Vert e^{\tau R}(x_{0}+x)-e^{\tau R}(x_{0}+\lambda)\right\Vert
+\left\Vert e^{\tau R}(x_{0}+\lambda)+\int_{0}^{\tau}e^{R(\tau-s)}B\left[
\omega_{0}(s)+\omega(s)\right]  ds-x_{0}\right\Vert \\
&  \leq\left\Vert e^{\tau R}x-e^{\tau R}\lambda\right\Vert +\left\Vert
\varphi(\tau,x_{0}+\lambda,\omega_{0}+\omega)-x_{0}\right\Vert \\
&  \leq ce^{(\rho_{j}+\delta)k\tau_{0}}\frac{b_{0}}{M_{j}(k\tau_{0})}%
\sqrt{d_{j}}\leq b_{0}.
\end{align*}
Hence we have proved that $\prod_{j=1}^{r}M_{j}(\tau)^{d_{j}}$ control
functions are sufficient to steer the system from all states in $x_{0}%
+\mathcal{C}$ back to $x_{0}+\mathcal{C}$ in time $\tau$. By the no-return
property of control sets it follows that the trajectories do not leave $D$
within the time interval $[0,\tau]$. By iterated concatenation of these
control functions we can construct an $(n\tau,K)$-spanning set $\mathcal{S}$
for each $n\in\mathbb{N}$ with cardinality%
\[
\left(  \prod_{j=1}^{r}M_{j}(\tau)^{d_{j}}\right)  ^{n}=\left(  \prod
_{j:\rho_{j}\geq0}(\left\lfloor e^{(\rho_{j}+\xi)\tau}\right\rfloor
+1)^{d_{j}}\right)  ^{n}.
\]
It follows that%
\begin{align*}
\log a_{n\tau}(f,K,Q)  &  \leq\log\sum_{\omega\in\mathcal{S}}e^{(S_{n\tau
}f)(\omega)}=\log\sum_{\omega\in\mathcal{S}}e^{\int_{0}^{n\tau}f(\omega
(t))dt}\\
&  =\log\sum_{\omega\in\mathcal{S}}e^{\int_{0}^{n\tau}f(\omega_{0}%
(t))dt+\int_{0}^{n\tau}[f(\omega(t))-f(\omega_{0}(t))]dt}\\
&  \leq\log\left[  \sum_{\omega\in\mathcal{S}}e^{\int_{0}^{n\tau}f(\omega
_{0}(t))dt}\cdot e^{\int_{0}^{n\tau}\varepsilon dt}\right]  .
\end{align*}
This implies
\begin{align*}
\frac{1}{n\tau}\log a_{n\tau}(f,K,Q)  &  \leq\frac{1}{\tau}\sum_{j:\rho
_{j}\geq0}d_{j}\log(\left\lfloor e^{(\rho_{j}+\xi)\tau}\right\rfloor
+1)+\frac{1}{n\tau}\int_{0}^{n\tau}f(\omega_{0}(t))dt+\varepsilon\\
&  \leq\frac{1}{\tau}\sum_{j:\rho_{j}\geq0}d_{j}\log(2e^{(\rho_{j}+\xi)\tau
})+\frac{1}{\tau_{0}}\int_{0}^{\tau_{0}}f(\omega_{0}(t))dt+\varepsilon\\
&  \leq\frac{d}{\tau}\log2+\frac{1}{\tau}\sum_{j:\rho_{j}\geq0}d_{j}(\rho
_{j}+\xi)\tau+\frac{1}{\tau_{0}}\int_{0}^{\tau_{0}}f(\omega_{0}%
(t))dt+\varepsilon\\
&  \leq\frac{d\xi}{\tau}+\sum_{j:\rho_{j}\geq0}d_{j}\rho_{j}+\frac{1}{\tau
_{0}}\int_{0}^{\tau_{0}}f(\omega_{0}(t))dt+2\varepsilon\\
&  <S_{0}+\frac{1}{\tau_{0}}\int_{0}^{\tau_{0}}f(\omega_{0}(t))dt+2\varepsilon
.
\end{align*}
Here we have also used (\ref{two}). Since $\varepsilon$ can be chosen
arbitrarily small and $S_{0}$ arbitrarily close to $\sum_{j=1}^{r}\max
(0,d_{j}\rho_{j})$, the assertion of the theorem follows.
\end{proof}

In order to see the relation to Floquet exponents the following simple result
is helpful.

\begin{proposition}
Consider a $\tau_{0}$-periodic solution of%
\[
\dot{x}(t)=Ax(t)+Bu(t).
\]
Then the Floquet exponents of the linearized system (linearized with respect
to $x$) are given by the real parts of the eigenvalues of $A$ and also the
algebraic multiplicities coincide. More generally, the Lyapunov exponents are
given by%
\[
\lim_{t\rightarrow\infty}\frac{1}{t}\log\left\Vert D_{x}\varphi
(t,x,u)y\right\Vert =\lim_{n\rightarrow\infty}\frac{1}{nT}\log\left\Vert
e^{AnT}x\right\Vert =\lambda,
\]
depending on $y$.
\end{proposition}

\begin{proof}
We have to analyze the eigenvalues of the linearization of the map
$x\mapsto\varphi(\tau_{0},x,u)=e^{A\tau_{0}}x+\int_{0}^{\tau_{0}}e^{A(\tau
_{0}-s)}Bu(s)ds$ given by $D_{x}\varphi(\tau_{0},x,u)=e^{A\tau_{0}}$. Thus the
assertion is a consequence of the spectral mapping theorem.
\end{proof}

This proposition shows that%
\[
\sum_{\lambda\in\sigma(A)}\max\{0,n_{\lambda}\mathrm{\operatorname{Re}%
}(\lambda)\}=\sum_{j=1}^{r}\max\{0,d_{j}\rho_{j}\},
\]
where $\rho_{1},\dotsc,\rho_{r}$ are the different Lyapunov exponents with
corresponding multiplicities of a periodic solution corresponding to a
periodic control. This is the term occurring in the estimate for the
invariance entropy in Kawan \cite[Theorem 5.1]{Kawa13}.

\begin{corollary}
\label{theorem20}Consider a linear control system of the form $\Sigma_{lin}$
and assume that the pair $(A,B)$ is controllable, that $A$ is hyperbolic and
the control range $U$ is a compact neighborhood of the origin. Let $D$ be the
unique control set, let $f\in C(U,\mathbb{R})$ and suppose that $\min
_{\omega\in U}f(\omega)=f(\omega_{0})$ with $\omega_{0}\in\mathrm{int}U$ and
there exists $x_{0}\in\mathrm{int}D$ with $Ax_{0}+B\omega_{0}=0$.

Then for every compact set $K\subset D$ with nonempty interior we have that
$(K,D)$ is an admissible pair and
\[
P_{inv}(f,K,D)=\sum_{\lambda\in\sigma(A)}\max\{0,n_{\lambda}\mbox{Re}(\lambda
)\}+f(\omega_{0}).
\]

\end{corollary}

\begin{proof}
This follows by Theorem \ref{linear_main}, since $(\omega_{0},x_{0})$ is a
(trivial) periodic solution in $\mathrm{int}U\times\mathrm{int}D$, and for
every $T$-periodic control $\omega(\cdot)$
\[
\frac{1}{T}\int_{0}^{T}f(\omega(s))ds\geq f(\omega_{0}).
\]

\end{proof}

\begin{example}
Consider the one-dimensional linear system given by the differential
equations
\[
\dot{x}(t)=ax(t)+\omega(t),\ \omega\in\mathcal{U},
\]
where $a>0$. We assume that the control range $U=[-1,1]$. Then the compact
interval $Q=\left[  -\frac{1}{a},\frac{1}{a}\right]  $ is the closure of the
unique control set with nonempty interior $D=\mathcal{O}^{-}(0)=\left(
-\frac{1}{a},\frac{1}{a}\right)  $ of this system.

Let $f\in C(U,\mathbb{R})$ such that $f(u_{0})=\inf f$ for some $u_{0}%
\in\mathrm{int}U$. Then $x_{0}:=-\frac{u_{0}}{a}\in\mathrm{int}D$ and
$(x_{0},u_{0})$ is an equilibrium pair. By Corollary \ref{theorem20} we have
\[
P_{inv}(f,K,Q)=\inf f+a.
\]

\end{example}

The next example (cf. Sontag \cite{Sont98}) presents an application of outer
invariance pressure to a mechanical control system and shows that, in this
case, this amount is related to the exponential growth rate of the total impulse
of external forces acting on the system.

\begin{example}
	
		Consider a pendulum to which one can apply a torque as an
					external force (see Fig. \ref{fig1}). We assume that friction is negligible, that all of the mass is concentrated at the end, and that the rod has unit length. From Newton’s law for rotating objects, there results, in terms of the variable $\alpha$ that describes the counter clockwise angle with respect to the vertical, the second-order nonlinear differential
			equation
			\begin{eqnarray}\label{ex24}
			m\ddot{\alpha}(t)+mg\sin(\alpha(t))=u(t),
			\end{eqnarray}
			where $m$ is the mass, $g$ the acceleration due to gravity, and $u(t)$ the value of the external torque at time $t$ (counter clockwise being positive).
		
		\begin{minipage}[t]{0.3\textwidth}
			\begin{tikzpicture}[baseline={([yshift={-1ex}]current bounding box.north)}]
			\tikzstyle{every node}+=[inner sep=0pt]
			\pgfmathsetmacro{\Gvec}{1.5}
			\pgfmathsetmacro{\myAngle}{30}
			\pgfmathsetmacro{\Gcos}{\Gvec*cos(\myAngle)}
			\pgfmathsetmacro{\Gsin}{\Gvec*sin(\myAngle)}
			\coordinate (centro) at (0,0);
			\phantom{\draw[dashed,black,-] (centro) -- ++ (330+\myAngle:3) node (mary) [black,below]{$ $};}
			\phantom{\draw[thick] (centro) -- ++(330+\myAngle:3) coordinate (bob);}
			\pic [draw, ->, "$u$", angle eccentricity=1.5] {angle = bob--centro--mary};
			\draw[dashed,black,-] (centro) -- ++ (0,-3.78) node (mary) [black,below]{$ $};
			\draw[thick] (centro) -- ++(280+\myAngle:3) coordinate (bob);
			\pic [draw, ->, "$\alpha$", angle eccentricity=1.3, angle radius=1cm] {angle = mary--centro--bob};
			\draw [-stealth] (bob) -- ($(bob)!\Gsin cm!90:(centro)$)
			coordinate (gsin)
			node[midway,above left] {$mg\sin\alpha$};
			\draw [-stealth] (bob) -- ++(0,-\Gvec)
			coordinate (g)
			node[near end,left] {$mg$};
			\filldraw [draw=black] (bob) circle[radius=0.15];
			\end{tikzpicture}
			
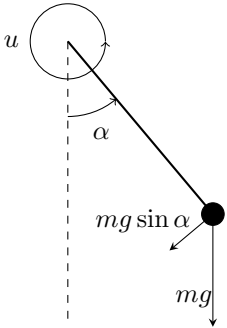
\captionof{figure}{Pendulum.}\label{fig1}
		\end{minipage}

The vertical stationary position $(\alpha,\dot{\alpha})=(\pi,0)$ is an equilibrium when the null control $\omega_{0}\equiv0$ is applied, but a small deviation from this will result in an unstable motion. Let us assume that our	objective is to apply torques as needed to correct for such deviations. For	small $\alpha-\pi$,
	\[
	\sin(\alpha)=-(\alpha-\pi)+r(\alpha-\pi),
	\]
	when $r(t)$ is a function which satisfies $\lim_{t\rightarrow0}\frac{r(t)}%
	{t}=0$.
	
	Since only small deviations are of interest, we drop the nonlinear part
	represented by the term $r(t)$. Thus, with $\gamma:=\alpha-\pi$ as a new
	variable, we replace equation (\ref{ex24}) by the linear differential
	equation
	\[
	m\ddot{\gamma}(t)-mg\gamma(t)=\omega(t).
	\]
	If we denote $x_{1}=\gamma$ and $x_{2}=\dot{\gamma}$, then we obtain%
	\[
	\Sigma_{1}~:~\left[
	\begin{array}
	[c]{c}%
	\dot{x}_{1}\\
	\dot{x}_{2}%
	\end{array}
	\right]  =\underset{=:A}{\underbrace{\left[
			\begin{array}
			[c]{cc}%
			0 & 1\\
			g & 0
			\end{array}
			\right]  }}\left[
	\begin{array}
	[c]{c}%
	x_{1}\\
	x_{2}%
	\end{array}
	\right]  \underset{=:B}{+\underbrace{\left[
			\begin{array}
			[c]{c}%
			0\\
			\frac{1}{m}%
			\end{array}
			\right]  }}\omega,~\omega(t)\in U:=[-\varepsilon,\varepsilon],\varepsilon
	>0\text{.}%
	\]
Note that the eigenvalues of $A$ are $\lambda_{\pm}=\pm\sqrt{g}$. System $\Sigma_{1}$ is via the (time-invariant) conjugacy map $(T,id_{U})$ conjugate to (cf. Example \ref{Example12})
	\[
	\Sigma_{2}~:~\left[
	\begin{array}
	[c]{c}%
	\dot{x_{1}}\\
	\dot{x_{2}}%
	\end{array}
	\right]  =\underbrace{\left[
		\begin{array}
		[c]{cc}%
		-\sqrt{g} & 0\\
		0 & \sqrt{g}%
		\end{array}
		\right]  }_{=:\widetilde{A}}\left[
	\begin{array}
	[c]{c}%
	x_{1}\\
	x_{2}%
	\end{array}
	\right]  +\underbrace{\left[
		\begin{array}
		[c]{c}%
		\frac{1}{2m}\\
		\frac{1}{2m}%
		\end{array}
		\right]  }_{=:\widetilde{B}}\omega,
	\]
because $\widetilde{A}=TAT^{-1}$ and $\widetilde{B}=TB$, where
	\[
	T=\frac{1}{2}\left[
	\begin{array}
	[c]{cc}%
	-\sqrt{g} & 1\\
	\sqrt{g} & 1
	\end{array}
	\right]  \text{ and }T^{-1}=\left[
	\begin{array}
	[c]{cc}%
	-\frac{1}{\sqrt{g}} & \frac{1}{\sqrt{g}}\\
	1 & 1
	\end{array}
	\right]  .
	\]
Note that $\widetilde{A}$ is hyperbolic and the pair $(\widetilde{A},\widetilde{B})$ is controllable. By Theorem \ref{Theorem_unique}, the unique control set $\tilde{D}$ with nonvoid interior of $\Sigma_{2}$ is
	\[
	\tilde{D}=\overline{\mathcal{O}^{+}(0)}\cap\mathcal{O}^{-}(0)=\left[
	-\frac{\varepsilon}{2m\sqrt{g}},\frac{\varepsilon}{2m\sqrt{g}}\right]
	\times\left(  -\frac{\varepsilon}{2m\sqrt{g}},\frac{\varepsilon}{2m\sqrt{g}%
	}\right)  .
	\]
Then the unique control set with nonvoid interior of $\Sigma_{1}$ is given by $D:=T(\tilde{D})$ and one computes%
	\[
	D=\left[  -d,d\right]  \times\left(  -d,d\right)  \text{ with }d:=\varepsilon
	\frac{\sqrt{g}+1}{2m\sqrt{g}}.
	\]
Here for a compact subset $K\subset Q:=D$ a $(\tau,K,Q)$-spanning set $\mathcal{S}$ represents a set of external torques $\omega$ that cause the angular position of the pendulum to remain in the interval $\left[
	-d,d\right]  $ and such that its angular velocity does not exceed $\left(
	-d,d\right)  $ when it starts in $K$. If $f(u)=\left\vert u\right\vert $,
	$u\in U=[-\varepsilon,\varepsilon]$, then $f\in C(U,\mathbb{R})$ and
	$0=f(0)=\inf f$. Note that here $(S_{\tau}f)(\omega)$ represents the impulse
	of the torque $\omega$ until time $\tau$. Hence, the invariance pressure
	$P_{inv}(f,K,D)$ measures the exponential growth rate of the amount of total
	impulse required of the external torques acting on the system to remain in $D$
	as time tends to infinity. Corollary \ref{theorem20} implies that
	$P_{inv}(f,K,D)=\sqrt{g}=h_{inv}(K,D)$. The reason is that within the control set $D$ one
	may steer the system from $K$ arbitrarily close to the equilibrium
	$0\in\mathbb{R}^{2}$ and keep it there with arbitrarily small torque.
\end{example}

\section*{Acknowledgments}
The authors appreciate the unknown referee’s valuable and profound comments.

\end{document}